\documentclass[a4paper,10pt, oneside]{article}
\usepackage{amsmath, amssymb, amsthm}
\usepackage[titletoc, title]{appendix}
\usepackage[pdftex, pdfpagelabels, bookmarks, hyperindex, hyperfigures,
colorlinks, pagebackref]{hyperref}
\usepackage[capitalise, nosort]{cleveref}

\crefname{equation}{}{}
\crefname{lem}{Lemma}{Lemmas}
\crefname{thm}{Theorem}{Theorems}

\DeclareMathOperator*{\Beta}{B}

\newcommand{\nm}[1]{\left\Vert #1 \right\Vert}
\newcommand{\snm}[1]{\left\vert #1 \right\vert}
\newcommand{\snmii}[1]
{
  \left\vert\kern-0.25ex
  \left\vert\kern-0.25ex
  \left\vert
  #1
  \right\vert\kern-0.25ex
  \right\vert\kern-0.25ex
  \right\vert
}

\newtheorem{assum}{Assumption}
\newtheorem{corollary}{Corollary}[section]
\newtheorem{lem}{Lemma}[section]
\newtheorem{rem}{Remark}[section]
\newtheorem{thm}{Theorem}[section]

\numberwithin{equation}{section}

\begin{document}
\title
{
  \Large\bf A new smoothness result for Caputo-type fractional ordinary
  differential equations
  \thanks
  {
    This work was supported in part by  Major Research Plan of National
    Natural Science Foundation of China (91430105) and National Natural Science Foundation
    of China (11401407).
  }
}
\author
{
  Binjie Li \thanks{Email: libinjiefem@yahoo.com},
  Xiaoping Xie \thanks{Corresponding author. Email: xpxie@scu.edu.cn},
  Shiquan Zhang \thanks{Email: shiquanzhang@scu.edu.cn}\\\\
  {School of Mathematics, Sichuan University, Chengdu 610064, China}
}
\date{}
\maketitle

\begin{abstract}
 We present a new smoothness result for Caputo-type fractional ordinary
 differential equations, which reveals that, subtracting a non-smooth function
 that can be obtained by the information available, a non-smooth solution
 belongs to $ C^m $ for some positive integer $ m $.

\end{abstract}

\section{Introduction}
Let us consider the following model problem: seek $ 0 < h \leqslant a $ and
\[
  y \in \left\{
    v \in C[0, h]:
    \nm{v-c_0}_{C[0, h]} \leqslant b
  \right\}
\]
such that
\begin{equation}
  \label{eq:model}
  \begin{cases}
    D_*^\alpha y = f(x, y), & 0 \leqslant x \leqslant h, \\
    y(0) = c_0, &
  \end{cases}
\end{equation}
where $ a > 0 $, $ b > 0 $, $ 0 < \alpha < 1 $, $ c_0 \in \mathbb R $, and
\[
  f \in C\left(
    [0, a] \times [c_0-b, c_0+b]
  \right).
\]
Above, the Caputo-type fractional differential operator $ D_*^\alpha: \, C[0, h]
\to C_0^{\infty}(0, h)' $ is given by
\begin{equation}
  D_*^\alpha z := D J^{1-\alpha} \big( z-z(0) \big)
\end{equation}
for all $ z \in C[0, h] $, where $ D $ denotes the well-known first order
generalized differential operator, and the Riemann-Liouville fractional integral
operator $ J^{1-\alpha}: C[0, h] \to C[0, h] $ is defined by
\[
  J^{1-\alpha} z (x) := \frac 1{ \Gamma(1-\alpha) }
  \int_0^x (x-t)^{-\alpha} z(t) \, \mathrm{d} t,
  \quad 0 \leqslant x \leqslant h,
\]
for all $ z \in C[0, h] $.

By \cite[Lemma 2.1]{Diethelm;2002}, the above problem is equivalent to seeking
solutions of the following Volterra integration equation:
\begin{equation}
  \label{eq:model2}
  y(x) = c_0 + \frac 1{\Gamma(\alpha)}
  \int_0^x (x-t)^{\alpha-1} f\big(t, y(t) \big) \,
  \mathrm{d}t.
\end{equation}
Diethelm and Ford~\cite{Diethelm;2002} proved that, if $ f $ is continuous, then
\cref{eq:model2} has a solution $ y \in C[0, h] $ for some $ 0 < h \leqslant a
$, and this solution is unique if $ f $ is Lipschitz continuous. A natural
question arises whether $ y $ can be smoother than being continuous. This is not
only of theoretical value, but also of great importance in developing numerical
methods for \cref{eq:model2}.

To this question, Miller and Feldstein~\cite{Miller;Feldstein;1971} gave the
first answer: if $f$ is analytic, then $ y $ is analytic in $ (0, h) $ for some
$ 0 < h \leqslant a $. Then Lubich \cite{Lubich;1983} considered the behavior of
the solution near $ 0 $. He showed that, if $ f $ is analytic at the origin,
then there exists a function $ Y $ of two variables that is analytic at the
origin such that
\[
  y(x) = Y(x, x^\alpha), \quad 0 \leqslant x \leqslant h,
\]
for some $ 0 < h \leqslant a $. The above work suggests that non-smoothness of
the solution to \cref{eq:model} is generally unavoidable. However,
Diethelm~\cite{Diethelm;2007} established a sufficient and necessary condition
under which $ y $ is analytic on $ [0, h] $ for some $ 0 < h \leqslant a $. But,
since we have already seen that non-smoothness of $ y $ is generally
unavoidable, it is not surprising that this condition is unrealistic. Recently,
Deng~\cite{Deng;2010} proposed two conditions: under the first condition the
solution belongs to $ C^m $ for some positive integer $ m $; under the second
one the solution is a polynomial. It should be noted that, the second condition
is just the one proposed in~\cite{Diethelm;2007}, and the first condition is
also unrealistic.

The main result of this paper is that, although the solution $ y $ of
\cref{eq:model} does not generally belong to $ C^m $ for some positive integer $
m $, we can still construct a non-smooth function of the form
\[
  S(x) := c_0 + \sum_{j=1}^n c_j x^{\gamma_j},
\]
such that
\[
  y - S \in C^m,
\]
provided $ f $ is sufficiently smooth. Most importantly, given $ c_0 $ and $ f
$, we can obtain $ S $ by a simple computation. This is significant in the
development of numerical methods for \cref{eq:model}. In addition, we obtain a
sufficient and necessary condition under which $ y \in C^m $. We note that this
condition is essentially the same as the first condition mentioned already
in~\cite[Theorem 2.8]{Deng;2010}, but the necessity was not considered therein.

The rest of this paper is organized as follows. In \cref{sec:notation} we
introduce some basic notation and preliminaries. In \cref{sec:main} we state the
main results of this paper, and present their proofs in \cref{sec:proof}.

\section{Notation and Preliminaries}
\label{sec:notation}
Let $ 0 < h < \infty $. We use $ C[0, h] $ to denote the space of all continuous
real functions defined on $ [0, h] $. For any $ k \in \mathbb N_{>0} $ and $ 0
\leqslant \gamma \leqslant 1 $, define
\begin{align}
  C^k[0, h] & := \left\{
    v \in C[0, h]: v^{(j)} \in C[0, h] \quad
    \text{for $ j = 1, 2, \dotsc, k $}
  \right\}, \\
  C^{k,\gamma}[0, h] & := \left\{
    v \in C^k[0, h]: \max_{ 0 \leqslant x < y \leqslant h }
    \snm{v}_{ C^{k,\gamma}[0, h] } < \infty
  \right\},
\end{align}
and endow the above two spaces with two norms respectively by
\begin{align}
  \nm{v}_{C^k[0, h]} & :=
  \max_{ 0 \leqslant j \leqslant k }
  \max_{ 0 \leqslant x \leqslant h }
  \snm{v^{(j)}(x)}
  & & \text{ for all $ v \in C^k[0, h] $,} \\
  \nm{v}_{ C^{k,\gamma}[0, h] } & :=
  \max\left\{
    \nm{v}_{C^k[0, h]}, \,
    \snm{v}_{ C^{k,\gamma}[0, h] }
  \right\}
  & & \text{ for all $ v \in C^{k,\gamma}[0, h] $ }.
\end{align}
Here the semi-norm $ \snm{\cdot}_{ C^{k,\gamma}[0, h] } $ is given by
\[
  \snm{v}_{ C^{k,\gamma}[0, h] } :=
  \sup_{0 \leqslant x < y \leqslant h}
  \frac{ \snm{v^{(k)}(x)-v^{(k)}(y)} }{ (y-x)^\gamma }
\]
for all $ v \in C^{k,\gamma}[0, h] $, and it is obvious that $ C^k[0, h] $
coincides with $ C^{k,0}[0, h] $.

For any $ s \in \mathbb N_{>0} $, define
\[
  \Lambda_s := \big\{
    \beta = (\beta_1, \beta_2, \dots, \beta_s) \in \{1,2\}^s
  \big\},
\]
and, for any $ \beta \in \Lambda_s $, we use the following notation:
\[
  \partial_\beta g :=
  \frac\partial{\partial x_{\beta_s}}
  \frac\partial{\partial x_{\beta_{s-1}}}
  \cdots
  \frac{\partial}{\partial x_{\beta_1}} g(x_1, x_2),
\]
where $ g $ is a real function of two variables. In addition, we define
\[
  \Lambda_0 := \left\{ \emptyset \right\},
\]
and denote by $ \partial_{\emptyset} $ the identity mapping.

\section{Main Results}
\label{sec:main}
Let us first make the following assumption on $ f $.
\begin{assum}
  There exist a positive integer $ n $, and a positive constant $ M $ such that
  \begin{align*}
    f \in C^n \left( [0, a] \times [c_0-b, c_0+b] \right), \\
    \max_{ (x, y) \in [0, a] \times [c_0-b, c_0+b] }
    \max_{
      \substack{
        0 \leqslant i \leqslant n \\
        0 \leqslant j \leqslant n \\
        i + j \leqslant n
      }
    }
    \snm{
      \frac{\partial^i}{ \partial x^i }
      \frac{\partial^j}{ \partial y ^j }
      f(x, y)
    } \leqslant M.
  \end{align*}
\end{assum}
\noindent Throughout this paper, we assume that the above assumption is
fulfilled.

Define $ J \in \mathbb N $ and a strictly increasing sequence $ \left\{ \gamma_i
\right\}_{i=1}^J $ by
\begin{equation}
  \left\{
    \gamma_j: 1 \leqslant j \leqslant J
  \right\} =
  \left\{
    i+j\alpha:
    i, j \in \mathbb N, \
    0 < i+j\alpha < m
  \right\},
\end{equation}
where
\begin{equation}
  m := \max\left\{ j \in \mathbb N: j < n\alpha \right\}.
\end{equation}
Define $ c_1 $, $ c_2 $, $ \dots $, $ c_J \in \mathbb R $ by
\begin{equation}
  \label{eq:I_1}
  Q(x) - S(x) + c_0 \in \text{span}
  \left\{
    x^{i+j\alpha}:
    i, j \in \mathbb N, \
    i+j\alpha \geqslant m
  \right\},
\end{equation}
where
\begin{equation}
  \label{eq:def-Q}
  Q(x) :=
  \sum_{s=0}^{n-1}
  \sum_{ \beta\in\Lambda_s }
  \frac{ \partial_\beta f(0, c_0) }{ \Gamma(\alpha) }
  \int_0^x (x-t_0)^{\alpha-1} \, \mathrm{d}t_0
  \prod_{k=1}^s \int_0^{ t_{k-1} }
  \frac{ 1+(-1)^{\beta_k+1} }2 +
  \frac{ 1+(-1)^{\beta_k} }2
  \sum_{j=1}^J \gamma_j c_j t_k^{\gamma_j-1} \, \mathrm{d}t_k,
\end{equation}
and
\begin{equation}
  \label{eq:def-S}
  S(x) := c_0 + \sum_{j=1}^J c_j x^{\gamma_j}.
\end{equation}
Above and throughout, a product of a sequence of integrals should be understood
in expanded form. For example, \cref{eq:def-Q} is understood by
\begin{align*}
  Q(x) :=
  \sum_{s=0}^{n-1}
  \sum_{ \beta\in\Lambda_s }
  \frac{ \partial_\beta f(0, c_0) }{ \Gamma(\alpha) }
  \int_0^x (x-t_0)^{\alpha-1} \, \mathrm{d}t_0
  & \int_0^{ t_0 }
  \frac{ 1+(-1)^{\beta_1+1} }2 +
  \frac{ 1+(-1)^{\beta_1} }2
  \sum_{j=1}^J \gamma_j c_j t_1^{\gamma_j-1} \, \mathrm{d}t_1 \\
  & \int_0^{ t_1 }
  \frac{ 1+(-1)^{\beta_2+1} }2 +
  \frac{ 1+(-1)^{\beta_2} }2
  \sum_{j=1}^J \gamma_j c_j t_2^{\gamma_j-1} \, \mathrm{d}t_2 \\
  & \dots \\
  & \int_0^{ t_{s-1} }
  \frac{ 1+(-1)^{\beta_s+1} }2 +
  \frac{ 1+(-1)^{\beta_s} }2
  \sum_{j=1}^J \gamma_j c_j t_s^{\gamma_j-1} \, \mathrm{d}t_s.
\end{align*}

\begin{rem}
  It is easy to see that we can express $ Q $ in the form
  \[
    Q(x) = \sum_{j=1}^L d_j x^{\gamma_j},
  \]
  where $ \left\{ \gamma_j \right\}_{j=J+1}^L $ is a strictly increasing
  sequence such that $ \gamma_J < \gamma_{J+1} $ and
  \[
    \left\{ \gamma_j: 1 \leqslant j \leqslant L \right\} =
    \left\{
      i+j\alpha: \,
      i, j\in \mathbb N, \,
      i \leqslant n-1, \, 1 \leqslant j \leqslant 1+(n-1)\gamma_J
    \right\}.
  \]
  Moreover, for $ 1 \leqslant j \leqslant J $, the value of $ d_j $ only depends
  on $ c_0 $, $ c_1 $, $ \dots $, $ c_{j-1} $, and $ f $ (more precisely, $
  \partial_\beta f(0, c_0) $, $ \beta \in \Lambda_s $, $ 1 \leqslant s \leqslant
  n-1 $). Obviously, there exist(s) uniquely $ c_1 $, $ c_2 $, $ \dots $, $ c_J
  $ such that \cref{eq:I_1} holds, and hence $ c_1 $, $ c_2 $, $ \dots $, $ c_J
  $ are/is well-defined. Furthermore, if $ \gamma_J + \alpha - m > 0 $, then
  \[
    Q - S \in C^{ m, \gamma_J+\alpha-m }[0, a];
  \]
  and if $ \gamma_J + \alpha -m = 0 $, then
  \[
    Q - S \in C^{ m, \alpha }[0, a].
  \]
\end{rem}

\begin{rem}
  Note that, $ S $ only depends on $ c_0 $ and
  \[
    \left\{
      \partial_\beta f(0, c_0): \,
      \beta \in \Lambda_s, \, 0 \leqslant s < n
    \right\}.
  \]
  Since $ c_0 $ and $ f $ are already available, we can obtain $ S $ by a simple
  calculation.
\end{rem}

Define
\[
  h^* := \min\left\{
    a,
    \left( \frac{ b\Gamma(1+\alpha) }M \right)^\frac1\alpha
  \right\}.
\]
By \cite[Theorem 2.2]{Diethelm;2002} we know that there exists a unique solution
$ y^* \in C[0, h^*] $ to \cref{eq:model}. Now we state the most important result
of this paper in the following theorem.
\begin{thm}
  \label{thm:main}
  There exist two positive constant $ C_0 $ and $ C_1 $ that only depends on $ a
  $, $ \alpha $ and $ M $, such that, for any $ 0 < h \leqslant h^* $ and $ K >
  0 $ such that
  \[
    \nm{(Q-S)'}_{ C^{m-1}[0, h] } + C_1 h^\alpha + C_0 h^\alpha \sum_{j=1}^m K^j
    \leqslant K,
  \]
  we have $ y^* - S \in C^m[0, h] $ and
  \begin{equation}
    \label{eq:main}
    \nm{(y^*-S)'}_{ C^{m-1}[0, h] } \leqslant K.
  \end{equation}
\end{thm}

\begin{corollary}
  \label{coro:main}
  There exists $ 0 < h \leqslant h^* $ such that $ y^* \in C^m[0, h] $ if, and
  only if,
  \begin{equation}
    \label{eq:cond-f}
    \frac{\partial^i}{ \partial x^i } f(0, c_0) = 0
    \quad \text{ for all $ 0 \leqslant i < m $. }
  \end{equation}
\end{corollary}
\begin{rem}
  \cref{coro:main} states that $ y^* \in C^1[0, h] $ for some $ 0 < h \leqslant
  h^* $ if and only if $ f(0, c_0) = 0 $. So we only have $ y^* \in C[0, h]
  \setminus C^1[0, h] $, if $ f(0, c_0) \neq 0 $. This yields great difficulty
  in developing high order numerical methods for \cref{eq:model}, although $ y^*
  \in C^m(0, h] $. Many numerical methods for \cref{eq:model} may not even
  converge theoretically, since they require that $ y^* \in C^m[0, h] $ for some
  positive integer $ m $. However, we can obtain the numerical values of $ y^* $
  at some left-most nodes by solving the following problem ($ y^* = y + S $):
  seek $ y \in C^m[0,\widetilde h] $ such that
  \[
    y(x) = c_0 - S(x) +
    \frac 1{ \Gamma(\alpha) } \int_0^x (x-t)^{\alpha-1} f\big( t, y(t)+S(t) \big)
    \, \mathrm{d}t, \quad 0 \leqslant x \leqslant \widetilde h,
  \]
  where $ \widetilde h \ll h $. Then we start the numerical methods for
  \cref{eq:model}.

\end{rem}

\begin{rem}
  Assuming that $ f $ satisfies $ f(x, c_0) = 0 $ for all $ 0 \leqslant x
  \leqslant a $, it is easy to see that
  \[
    c_i = 0 \quad \text{for all $ 1 \leqslant i \leqslant J $},
  \]
  and hence $ S = c_0 $. Then \cref{thm:main} implies $ y^* \in C^m[0, h] $.
  Actually, in this case, it is easy to see that $ y^* = c_0 $.
\end{rem}

\begin{rem}
  Put
  \[
    \Theta := \left\{
      1 \leqslant j \leqslant J:
      \gamma_j \not \in \mathbb N
    \right\}.
  \]
  Obviously,
  \[
    \sum_{ j \in \Theta } c_j x^{\gamma_j}
  \]
  is the singular part (compared to the $ C^m $ regularity) in $ S $, and thus
  the singular part in $ y^* $. \cref{coro:main} essentially claims that
  \cref{eq:cond-f} holds if and only if $ c_j = 0 $ for all $ j \in \Theta $.
  Since \cref{eq:cond-f} is rare, we can consider singularity as an intrinsic
  property of solutions to fractional differential equations. In addition, we
  have the following result: that $ c_j = 0 $ for all $ 1 \leqslant j \leqslant
  J $ is equivalent to that $ c_j = 0 $ for all $ j \in \Theta $. This is
  contained in the proof of \cref{coro:main} in \cref{ssec:coro}.
\end{rem}

\section{Proofs}
\label{sec:proof}
Let $ 0 < h < \infty $. For any $ k \in \mathbb N $ and $ \gamma \in [0,1] $,
define
\begin{align}
  \mathcal C^{k,\gamma}[0, h] & :=
  \left\{
    v \in C^{k,\gamma}[0, h]:
    v^{(j)}(0) = 0,\quad j = 0, 1, 2, \dots, k
  \right\}, \\
  \widehat{\mathcal C}^{k,\gamma}[0, h] & :=
  \left\{
    v \in \mathcal C^{k,\gamma}[0, h]:
    \nm{v+S-c_0}_{C[0, h]} \leqslant b
  \right\}.
\end{align}
In particular, we use $ \mathcal C^k[0, h] $ and $ \widehat{\mathcal C}^k[0, h]
$ to abbreviate $ \mathcal C^{k,0}[0, h] $ and $ \widehat{\mathcal C}^{k,0}[0,
h] $ respectively for $ k \in \mathbb N_{>0} $, and use $ \mathcal C[0, h] $ and
$ \widehat{\mathcal C}[0, h] $ to abbreviate $ \mathcal C^0[0, h] $ and $
\widehat{\mathcal C}^0[0, h] $ respectively. In addition, for a function $ v $
defined on $ (0, h] $ with $ h > 0 $, by $ v \in \mathcal C^{k,\gamma}[0, h] $
we mean that, setting $ v(0) := 0 $, the function $ v $ belongs to $ \mathcal
C^{k,\gamma}[0, h] $.

In the remainder of this paper, unless otherwise specified, we use $ C $ to
denote a positive constant that only depends on $ \alpha $, $ a $ and $ M $, and
its value may differ at each occurrence. By the definitions of $ c_1 $, $ c_2 $,
$ \dots $, $ c_J $, it is easy to see that $ \snm{c_j} \leqslant C $ for all $ 1
\leqslant j \leqslant J $, and we use this implicitly in the forthcoming
analysis.

\subsection{Some Auxiliary Results}
We start by introducing some operators. For $ 0 < h \leqslant a $, define $
\mathcal P_{1, h}: \widehat{\mathcal C}^m[0, h] \to \mathcal C[0, h] $, $ \mathcal
P_{2, h}: \widehat{\mathcal C}^m[0, h] \to \mathcal C[0, h] $, and $ \mathcal
P_{3, h}: \widehat{\mathcal C}^m[0, h] \to \mathcal C[0, h] $, respectively, by
\begin{align}
  \mathcal P_{1, h} z(x) & :=
  \frac1{\Gamma(\alpha)}
  \int_0^x (x-t)^{\alpha-1} \mathcal G_{1, h} z(t) \, \mathrm{d}t,
  \label{eq:def-P_1} \\
  \mathcal P_{2, h}z(x) & :=
  \frac1{\Gamma(\alpha)}
  \int_0^x (x-t)^{\alpha-1} \mathcal G_{2, h} z(t) \, \mathrm{d}t,
  \label{eq:def-P_2} \\
  \mathcal P_{3, h} z(x) & :=
  \frac1{\Gamma(\alpha)}
  \int_0^x (x-t)^{\alpha-1} \mathcal G_{3, h} z(t) \, \mathrm{d}t,
  \label{eq:def-P_3}
\end{align}
for all $ z \in \widehat{\mathcal C}^m[0, h] $, where $ \mathcal G_{1, h} z $, $
\mathcal G_{2, h} z $, $ \mathcal G_{3, h} z \in \mathcal C[0, h] $ are given
respectively by
\begin{align}
  \mathcal G_{1, h} z(t_0) &:=
  \sum_{s=1}^n \sum_{ \substack{\beta\in\Lambda_s \\ \beta_s=2} }
  \prod_{k=1}^{s-1} \int_0^{t_{k-1}}
  \frac{ 1+(-1)^{\beta_k+1} } 2 +
  \frac{ 1+(-1)^{\beta_k} } 2 \sum_{j=1}^J \gamma_j c_j t_k^{\gamma_j-1}
  \,\mathrm{d} t_k \notag \\
  & \qquad\qquad\qquad\quad \int_0^{t_{s-1}} z'(t_s) \partial_\beta f
  \big( t_s, z(t_s)+S(t_s) \big) \, \mathrm{d}t_s,
  \label{eq:def-G_1} \\
  \mathcal G_{2, h} z(t_0) &:=
  \sum_{ \substack{\beta\in\Lambda_n\\\beta_n=2} }
  \prod_{k=1}^{n-1} \int_0^{t_{k-1}}
  \frac{1+(-1)^{\beta_k+1}}2 +
  \frac{ 1+(-1)^{\beta_k} }2 \sum_{j=1}^J \gamma_j c_j t_k ^{\gamma_j -1}
  \,\mathrm{d}t_k \notag \\
  & \qquad\qquad\qquad \int_0^{t_{n-1}} \partial_\beta f\big( t_n , z(t_n)+S(t_n) \big)
  \sum_{j=1}^J \gamma_j c_jt_n^{\gamma_j-1} \, \mathrm{d}t_n,
  \label{eq:def-G_2} \\
  \mathcal G_{3, h} z(t_0) &:=
  \sum_{ \substack{ \beta \in \Lambda_n \\ \beta_n =1 } }
  \prod_{k=1}^{n-1} \int_0^{t_{k-1}}
  \frac{ 1+(-1)^{\beta_k +1} }2 +
  \frac{ 1+(-1)^{\beta_k} }2 \sum_{j=1}^J \gamma_j c_j t_k ^{\gamma_j -1}
  \,\mathrm{d} t_k \notag \\
  & \qquad\qquad\qquad \int_0^{t_{n-1}} \partial_\beta f
  \big( t_n, z(t_n)+S(t_n) \big) \, \mathrm{d}t_n,
  \label{eq:G_3}
\end{align}
for all $ 0 \leqslant t_0 \leqslant h $.

Then let us present the following important results for the above operators.
\begin{lem}
  \label{lem:basic}
  Let $ 0 < h \leqslant a $. For any $ z \in \widehat{\mathcal C}^m [0, h] $,
  we have
  \begin{equation}
    \label{eq:basic}
    \frac1{ \Gamma(\alpha) } \int_0^x (x-t)^{\alpha-1}
    f\big( t, z(t)+S(t) \big) \, \mathrm{d}t =
    Q(x) + \mathcal P_{1, h}z(x) + \mathcal P_{2, h}z(x) +
    \mathcal P_{3, h}z(x)
  \end{equation}
  for all $ 0 \leqslant x \leqslant h $.
\end{lem}
\begin{proof}
  Let $ \beta \in \Lambda_s $ with $ 1 \leqslant s < n $. For any $ 0 < t_s
  \leqslant h $, applying the fundamental theorem of calculus yields
  \begin{align*}
    \partial_\beta f \big( t_s , z(t_s)+S(t_s) \big) = {} &
    \partial_\beta f \big( \epsilon, z(\epsilon) + S(\epsilon) \big) +
    \int_\epsilon^{t_s} \partial_{ \widetilde \beta } f \big(
      t_{s+1}, z(t_{s+1})+S(t_{s+1})
    \big) \, \mathrm{d} t_{s+1} + {} \\
    & \int_\epsilon^{t_s} \left(
      z'(t_{s+1}) +
      \sum_{j=1}^J \gamma_j c_j t_{s+1}^{\gamma_j-1}
    \right)
    \partial_{ \overset{\thickapprox}\beta } f\big(
      t_{s+1}, z(t_{s+1})+S(t_{s+1})
    \big) \, \mathrm{d} t_{s+1}
  \end{align*}
  for all $ 0 < \epsilon \leqslant t_s $, where $ \widetilde\beta :=
  (\beta_1,\beta_2, \dotsc, \beta_s, 1) $ and $ \overset{ \thickapprox }\beta :=
  ( \beta_1,\beta_2,\dotsc,\beta_s,2 ) $. Taking limits on both sides of the
  above equation as $ \epsilon $ approaches $ {0+} $, we obtain
  \begin{align*}
    \partial_\beta f\big( t_s, z(t_s)+S(t_s) \big) = {} &
    \partial_\beta f\big( 0, c_0) \big) +
    \int_0^{t_s} \partial_{\widetilde\beta} f\big(
      t_{s+1}, z(t_{s+1})+S(t_{s+1})
    \big) \, \mathrm{d} t_{s+1} + {} \\
    & \int_0^{t_s} \left(
      z'(t_{s+1}) +
      \sum_{j=1}^J \gamma_j c_j t_{s+1}^{\gamma_j-1}
    \right)
    \partial_{ \overset{\thickapprox}\beta } f\big(
      t_{s+1}, z(t_{s+1})+S(t_{s+1})
    \big) \, \mathrm{d} t_{s+1}.
  \end{align*}
  Using this equality repeatedly, we easily obtain \cref{eq:basic}. This
  completes the proof.
\end{proof}

\begin{lem}
  \label{lem:P_1}
  Let $ 0 < h \leqslant a $. For any $ z \in \widehat{\mathcal C}^m[0, h] $, we
  have $ \mathcal P_{1, h} z \in \mathcal C^{m,\alpha}[0, h] $ and
  \begin{align}
    \nm{ (\mathcal P_{1, h} z)' }_{ C^{m-1}[0, h] }
    \leqslant Ch^\alpha \sum_{j=1}^m \nm{z'}_{ C^{m-1}[0, h] }^j,
    \label{eq:P_1-1} \\
    \snm{ (\mathcal P_{1, h} z)^{(m)} }_{ C^{0,\alpha}[0, h] }
    \leqslant C \sum_{j=1}^m \nm{z'}_{C^{m-1}[0, h]}^j.
    \label{eq:P_1-2}
  \end{align}
\end{lem}

\begin{lem}
  \label{lem:P_2-P_3}
  Let $ 0 < h \leqslant a $. For any $ z \in \widehat{\mathcal C}^m[0, h] $, we
  have $ \mathcal P_{2, h} z$, $ \mathcal P_{3, h} z \in \mathcal
  C^{m,\alpha}[0, h] $ and
  \begin{align}
    \nm{ ( \mathcal P_{2, h}z )' }_{ C^{m-1}[0, h] } +
    \nm{ ( \mathcal P_{3, h}z )' }_{ C^{m-1}[0, h] }
    \leqslant C h^\alpha,
    \label{eq:P_2-P_3-1} \\
    \snm{ ( \mathcal P_{2, h}z )^{(m)} }_{ C^{0,\alpha}[0, h] } +
    \snm{ ( \mathcal P_{3, h}z )^{(m)} }_{ C^{0,\alpha}[0, h] }
    \leqslant C.
    \label{eq:P_2-P_3-2}
  \end{align}
\end{lem}

To prove the above two lemmas, we need several lemmas below.

\begin{lem}
  \label{lem:basic-3}
  Let $ 0 < h \leqslant a $ and $ g \in \mathcal C^m[0, h] $. We have $ w \in
  \mathcal C^{m,\alpha}[0, h] $ and
  \begin{align}
    \nm{w'}_{C^{m-1}[0, h]} & \leqslant Ch^\alpha \nm{g'}_{C^{m-1}[0, h]},
    \label{eq:basic-3-1} \\
    \snm{w^{(m)}}_{C^{0,\alpha}[0, h]} & \leqslant C \nm{g^{(m)}}_{C[0, h]},
    \label{eq:basic-3-2}
  \end{align}
  where
  \[
    w(x) := \int_0^x (x-t)^{\alpha-1} g(t)\,\mathrm{d}t,
    \quad 0 \leqslant x \leqslant h.
  \]
\end{lem}
\begin{proof}
  Since $ g \in \mathcal C^m[0, h] $ we have
  \[
    w^{(i)}(x) = \int_0^x (x-t)^{\alpha-1} g^{(i)}(t) \, \mathrm{d}t,
    \quad 1 \leqslant i \leqslant m.
  \]
  Then $ w \in \mathcal C^m[0, h] $ and \cref{eq:basic-3-1} follow, and
  \cref{eq:basic-3-2} follows from \cite[Theorem 3.1]{Samko;1993}.  This
  completes the proof.
\end{proof}


\begin{lem}
  \label{lem:lxy-3}
  Let $ 0 < h \leqslant a $, and $ k $, $ l \in \mathbb N $ such that $ k
  \leqslant m $ and $ l\alpha \leqslant 1 $. For any $ g \in \mathcal
  C^{k, l\alpha}[0, h] $, define
  \[
    w(x) := \int_0^x \sum_{j=1}^J \gamma_j c_j t^{\gamma_j-1} g(t)
    \, \mathrm{d}t,
    \quad 0 < x \leqslant h.
  \]
  Then we have the following results:
  \begin{itemize}
    \item If $ (l+1)\alpha \leqslant 1 $, then we have
      $ w \in \mathcal C^{k,(l+1)\alpha}[0, a] $ and
      \[
        \nm{w}_{ \mathcal C^{ k,(l+1)\alpha } } \leqslant C
        \nm{g}_{ \mathcal C^{k, l\alpha} }.
      \]
    \item If $ (l+1)\alpha > 1 $, then we have
      $ w \in \mathcal C^{ k+1,(l+1)\alpha-1 }[0, a] $ and
      \[
        \nm{w}_{ \mathcal C^{ k+1,(l+1)\alpha-1 } } \leqslant C
        \nm{g}_{ \mathcal C^{k, l\alpha} }.
      \]
  \end{itemize}
\end{lem}

For any $ 0 < h \leqslant a $, $ w \in \mathcal C[0, h] $, and $ \beta \in
\Lambda_s $ with $ 1 \leqslant s \leqslant n $, define $ \mathcal T_{w,\beta, h}:
\widehat{\mathcal C}^m[0, h] \to \mathcal C[0, h] $ by
\[
  \mathcal T_{w,\beta, h} z(x) := w(x) \partial_\beta f\big( x, z(x)+S(x) \big),
\]
for all $ z \in \widehat{\mathcal C}^m[0, h] $.
\begin{lem}
  \label{lem:lxy-4}
  For $ 0 \leqslant k \leqslant m $, we have $ \mathcal T_{w,\beta, h}z \in
  \mathcal C^{\min\{k, n-s\}}[0, h] $ and
  \begin{equation}
    \label{eq:lxy-4-1}
    \nm{ \mathcal T_{w,\beta, h}z }_{ C^{ \min \{ k, n-s \} }[0, h] }
    \leqslant C \nm{w}_{C^k[0, h]}
    \sum_{j=0}^{ \min \{ k, n-s \} }\nm{z'}_{ C^{m-1}[0, h] }^j
  \end{equation}
  for all $ 0 < h \leqslant a $, $ w \in \mathcal C^k[0, h] $, $ \beta \in
  \Lambda_s $ with $ 1 \leqslant s \leqslant n $, and $ z \in \widehat{\mathcal
  C}^m[0, h] $.
\end{lem}
\noindent The proofs of \cref{lem:lxy-3,lem:lxy-4} are presented in
\cref{sec:append-proofs}. In the rest of this subsection, we give the proofs of
\cref{lem:P_1,lem:P_2-P_3}.

\noindent {\bf Proof of \cref{lem:P_1}.}
  By \cref{eq:def-P_1}, \cref{eq:def-G_1}, and \cref{lem:basic-3}, it suffices
  to show that, for each $ \beta \in \Lambda_s $ with $ \beta_s = 2 $, we have $
  g_0 \in \mathcal C^m[0, h] $ and
  \begin{equation}
    \label{eq:P_1-3}
    \nm{g_0}_{C^m[0, h]} \leqslant
    C \sum_{j=1}^{ \min \{ m, n-s+1 \} } \nm{z'}_{ C^{m-1}[0, h] }^j,
  \end{equation}
  where, if $ s = 1 $, then
  \[
    g_0(x) := \int_0^x z'(t) \partial_2 f\big( t, z (t)+S(t) \big) \, \mathrm{d}t;
  \]
  if $ 2 \leqslant s \leqslant n $, then
  \begin{align*}
    g_0(x) & := \int_0^x \left(
      \frac{ 1+(-1)^{\beta_1+1} }2 +
      \frac{ 1+(-1)^{\beta_1} }2
      \sum_{j=1}^J \gamma_j c_j t^{\gamma_j-1}
    \right)
    g_1(t)\,\mathrm{d}t, \\
    g_1(x) & := \int_0^x \left(
      \frac{ 1+(-1)^{\beta_2+1} }2 +
      \frac{ 1+(-1)^{\beta_2} }2
      \sum_{j=1}^J \gamma_j c_j t^{\gamma_j-1}
    \right)
    g_2(t) \, \mathrm{d}t, \\
    & \vdots \\
    g_{s-2}(x) &:= \int_0^x \left(
      \frac{ 1+(-1)^{ \beta_{s-1}+1 } }2 +
      \frac{ 1+(-1)^{\beta_{s-1}} }2
      \sum_{j=1}^J \gamma_j c_jt^{\gamma_j-1}
    \right)
    g_{s-1}(t)\,\mathrm{d}t, \\
    g_{s-1}(x) & := \int_0^x z'(t)
    \partial_\beta f\big( t, z(t)+S(t) \big) \, \mathrm{d}t.
  \end{align*}

  To do so, we proceed as follows. If $ s=1 $, then by \cref{lem:lxy-4} we obtain $
  g_0 \in \mathcal C^m[0, h] $ and \cref{eq:P_1-3}. Let us suppose that $ 2
  \leqslant s \leqslant n $. By \cref{lem:lxy-4} it follows $ g_{s-1} \in \mathcal
  C^{ \min\{ m, n-s+1 \} }[0, h] $ and
  \[
    \nm{g_{s-1}}_{ C^{ \min \{ m, n-s+1 \} }[0, h] }
    \leqslant C \sum_{j=1}^{ \min \{ m, n-s+1 \} }
    \nm{z'}_{ C^{m-1}[0, h] }^j.
  \]
  Then, by the simple estimate
  \[
    (n-s+1) + (s-1)\alpha > m,
  \]
  applying \cref{lem:lxy-3} to $ g_{s-2} $, $ g_{s-3} $, $ \dots $, $ g_0 $
  successively yields $ g_0 \in \mathcal C^m[0, h] $ and \cref{eq:P_1-3}. This
  completes the proof of \cref{lem:P_1}.
\hfill\ensuremath{\blacksquare}

\noindent{\bf Proof of \cref{lem:P_2-P_3}.}
  Let us first show that $ \mathcal G_{2, h} z \in \mathcal C^m[0, h] $ and
  \begin{equation}
    \label{eq:P_2-P_3-3}
    \nm{ \mathcal G_{2, h} z }_{C^m[0, h]} \leqslant C.
  \end{equation}
  By \cref{eq:def-G_2} it suffices to show that, for any $ \beta \in \Lambda_n $
  with $ \beta_n = 2 $, we have $ g_0 \in \mathcal C^m[0, h] $ and
  \begin{equation}
    \label{eq:P_2-P_3-4}
    \nm{g_0}_{C^m[0, h]} \leqslant C,
  \end{equation}
  where
  \begin{align*}
    g_0(x) & := \int_0^x \left(
      \frac{ 1+(-1)^{\beta_1+1} }2 +
      \frac{ 1+(-1)^{\beta_1} }2 \sum_{j=1}^J \gamma_j c_j t^{\gamma_j-1}
    \right) g_1(t)\,\mathrm{d}t, \\
    g_1(x) & := \int_0^x \left(
      \frac{ 1+(-1)^{\beta_2+1} }2 +
      \frac{ 1+(-1)^{\beta_2} }2 \sum_{j=1}^J \gamma_j c_j t^{\gamma_j-1}
    \right) g_2(t)\,\mathrm{d}t, \\
    & \vdots \\
    g_{n-2}(x) &:= \int_0^x \left(
      \frac{ 1+(-1)^{ \beta_{s-1}+1 } }2 +
      \frac{ 1+(-1)^{\beta_{s-1}} }2 \sum_{j=1}^J \gamma_j c_j t^{\gamma_j-1}
    \right) g_{s-1}(t)\,\mathrm{d}t, \\
    g_{n-1}(x) & := \int_0^x
    \partial_\beta f\big( t, z(t)+S(t) \big)
    \sum_{j=1}^J \gamma_j c_jt^{\gamma_j -1} \, \mathrm{d}t,
  \end{align*}
  for all $ 0 \leqslant x \leqslant h $. Noting the fact that
  \[
    \partial_\beta f\big( \cdot, z(\cdot)+S(\cdot) \big) \in C[0, h],
  \]
  and $ \gamma_j \geqslant \alpha $ for all $ 1 \leqslant j \leqslant J $, we
  easily obtain $ g_{n-1} \in \mathcal C^{0,\alpha}[0, h] $ and
  \[
    \nm{g_{n-1}}_{ C^{0,\alpha}[0, h] } \leqslant C.
  \]
  Then, applying \cref{lem:lxy-3} to $ g_{n-2} $, $ g_{n-3} $, $ \dots $, $g_0$
  successively, and using the fact $ n\alpha > m $, we obtain $ g_0 \in \mathcal
  C^m[0, h] $ and \cref{eq:P_2-P_3-4}. Thus we have showed $ \mathcal G_{2, h} z
  \in \mathcal C^m[0, h] $ and \cref{eq:P_2-P_3-3}.

  Similarly, we can show that $ \mathcal G_{3, h} z \in \mathcal C^m[0, h] $ and $
  \nm{ \mathcal G_{3, h} z }_{C^m[0, h]} \leqslant C $. Consequently, by
  \cref{eq:def-P_2}, \cref{eq:def-P_3}, and \cref{lem:basic-3}, we infer that $
  \mathcal P_{2, h} z $, $ \mathcal P_{3, h} z \in \mathcal C^{m,\alpha}[0, h] $,
  and \cref{eq:P_2-P_3-1,eq:P_2-P_3-2} hold. This completes the proof.
\hfill\ensuremath{\blacksquare}

\subsection{ Proof of \texorpdfstring{\cref{thm:main}}{} }
By \cref{lem:P_1,lem:P_2-P_3} there exist two positive constants $ C_0 $ and $
C_ 1 $ that only depend on $ a $, $ \alpha $ and $ M $, such that
\begin{equation}
  \label{eq:C_0}
  \nm{ ( \mathcal P_{1, h} z )' }_{ C^{m-1}[0, h] }
  \leqslant C_0 h^\alpha \sum_{j=1}^m \nm{z'}_{ C^{m-1}[0, h] }^j,
\end{equation}
\begin{equation}
  \label{eq:C_1}
  \nm{ ( \mathcal P_{2, h}z )' }_{ C^{m-1}[0, h] } +
  \nm{ ( \mathcal P_{3, h}z )' }_{ C^{m-1}[0, h] }
  \leqslant C_1 h^\alpha,
\end{equation}
for all $ 0 < h \leqslant a $ and $ z \in \widehat{\mathcal C}^m[0, h] $. Let $
0 < h \leqslant h^* $ and $ K > 0 $ such that
\begin{equation}
  \label{eq:h-K} \nm{(Q-S)'}_{ C^{m-1}[0, h] } + C_1 h^\alpha + C_0 h^\alpha
  \sum_{j=1}^m K^j \leqslant K.
\end{equation}
Define $ \mathcal J: V \to \mathcal C[0, h] $ by
\begin{equation}
  \label{eq:J}
  \mathcal Jz (x) := c_0 - S(x) +
  \frac1{ \Gamma(\alpha) } \int_0^x
  (x-t)^{\alpha-1} f\big( t, z(t)+S(t) \big) \, \mathrm{d}t,
\end{equation}
for all $ z \in V $ and $ x \in [0, h] $, where
\begin{equation}
  V := \left\{
    v \in \widehat{\mathcal C}^m[0, h]:
    \nm{v'}_{ C^{m-1}[0, h] } \leqslant K
  \right\}.
\end{equation}
\begin{rem}
  It is clear that $ V $ is a bounded, closed, convex subset of $ C^m[0, h] $.
\end{rem}
\begin{rem}
  Let $ \delta > 0 $. If we put
  \begin{align*}
    K & := \nm{(Q-S)'}_{ C^{m-1}[0, h] } + C_1 a^\alpha + \delta, \\
    h &:= \min\left\{
      h^*, \,
      \left( \delta^{-1} C_0 \sum_{j=1}^m K^j \right)^{ -\frac1\alpha }
    \right\},
  \end{align*}
  then \cref{eq:h-K} holds.
\end{rem}

For the operator $ \mathcal J $, we have the following key result.
\begin{lem}
  \label{lem:J}
  For each $ z \in V $, we have $ \mathcal Jz \in V $ and
  \begin{equation}
    \label{eq:J-1}
    \snm{ (\mathcal Jz)^{(m)} }_{ C^{0,\gamma}[0, h] }
    \leqslant \snm{(Q-S)^{(m)}}_{ C^{0,\gamma}[0, h] } +
    C \sum_{j=0}^m K^j,
  \end{equation}
  where $ \gamma := \alpha $ if $ \gamma_J + \alpha = m $, and $ \gamma :=
  \gamma_J + \alpha - m $ if $ \gamma_J + \alpha > m $.
\end{lem}
\begin{proof}
  Let us first show $ \mathcal Jz \in V $. Using \cref{eq:J} and the fact $ h
  \leqslant \left( \frac{ b \Gamma(1+\alpha) } M \right)^{ \frac1\alpha } $, we
  have
  \begin{align*}
    \snm{ \mathcal Jz(x)+S(x)-c_0 } =
    \frac1{\Gamma(\alpha)}
    \snm{
      \int_0^x (x-t)^{\alpha-1} f \big( t, z(t)+S(t) \big) \, \mathrm{d}t
    } \leqslant
    \frac{ Mh^\alpha }{ \Gamma(1+\alpha) } \leqslant b
  \end{align*}
  for all $ x \in [0, h] $, and so
  \[
    \nm{ \mathcal Jz +S - c_0 }_{C[0, h]} \leqslant b.
  \]
  By \cref{lem:basic} we have
  \begin{equation}
    \label{eq:J-2}
    \mathcal Jz (x) = c_0 - S(x) + Q(x) + \mathcal P_{1, h} z (x) +
    \mathcal P_{2, h} z(x) + \mathcal P_{3, h} z(x),
  \end{equation}
  and then, by \cref{lem:P_1,lem:P_2-P_3}, and the fact $ c_0 - S + Q \in
  \mathcal C^m[0, h] $, we obtain $ \mathcal J z \in \mathcal C^m[0, h] $. It
  remains, therefore, to show that
  \begin{equation}
    \label{eq:J-3}
    \nm{(\mathcal Jz)'}_{C^{m-1}[0, h]} \leqslant K.
  \end{equation}
  To this end, note that, by \cref{eq:J-2,eq:C_0,eq:C_1} we obtain
  \[
    \nm{ (\mathcal Jz)' }_{ C^{m-1}[0, h] } \leqslant
    \nm{(Q-S)'}_{ C^{m-1}[0, h] } + C_1 h^\alpha +
    C_0 h^\alpha \sum_{j=1}^m K^j,
  \]
  and then \cref{eq:J-3} follows from \cref{eq:h-K}. We have thus showed $
  \mathcal Jz \in V $.

  Finally, let us show \cref{eq:J-1}. By \cref{lem:P_1,lem:P_2-P_3} we obtain
  \[
    \snm{(\mathcal P_{1, h} z)^{(m)}}_{C^{0,\alpha}[0, h]} +
    \snm{(\mathcal P_{2, h}z)^{(m)}}_{C^{0,\alpha}[0, h]} +
    \snm{(\mathcal P_{3, h}z)^{(m)}}_{C^{0,\alpha}[0, h]}
    \leqslant C \sum_{j=0}^m \nm{z'}_{C^{m-1}[0, h]}^j
    \leqslant C \sum_{j=0}^m K^j.
  \]
  From the fact $ \gamma \leqslant \alpha $ it follows
  \[
    \snm{(\mathcal P_{1, h} z)^{(m)}}_{C^{0,\gamma}[0, h]} +
    \snm{(\mathcal P_{2, h}z)^{(m)}}_{C^{0,\gamma}[0, h]} +
    \snm{(\mathcal P_{3, h}z)^{(m)}}_{C^{0,\gamma}[0, h]}
    \leqslant C \sum_{j=0}^m K^j.
  \]
  Using this estimate and the fact that $ (Q-S)^{(m)} \in C^{0,\gamma} $ by the
  definitions of $ Q $ and $ S $, the desired estimate \cref{eq:J-1} follows
  from \cref{eq:J-2}. This completes the proof.
\end{proof}

By the famous Arzel\`a-Ascoli Theorem and \cref{lem:J}, it is evident that $
\mathcal J: V \to V $ is a compact operator, where $ V $ is endowed with norm $
\nm{\cdot}_{C^m[0, h]} $. Therefore, since $ V $ is a bounded, closed, convex
subset of $ C^m[0, h] $, using the Schauder Fixed-Point Theorem gives that there
exists $ z \in V $ such that
\[
  \mathcal J z = z.
\]
Putting
\[
  y(x) := z(x) + S(x), \quad 0 \leqslant x \leqslant h,
\]
we obtain
\[
  y(x) = c_0 + \frac1{ \Gamma(\alpha) } \int_0^x (x-t)^{\alpha-1}
  f\big( t, y(t) \big) \, \mathrm{d}t, \quad 0 \leqslant x \leqslant h.
\]
By \cite[Lemma 2.1]{Diethelm;2002}, the above $ y $ is a solution of
\cref{eq:model}, and then, since $ y^* $ is the unique solution of
\cref{eq:model} on $ [0, h^*] $, we have $ y^* = y $ on $ [0, h] $. Therefore,
it is obvious that $ y^* - S \in C^m[0, h] $ and \cref{eq:main} hold. This
completes the proof of \cref{thm:main}.


\subsection{ Proof of \texorpdfstring{\cref{coro:main}}{} }
\label{ssec:coro}
Let us first state the following fact. For each $ 1 \leqslant j \leqslant J $,
by the definition of $ c_j $, a straightforward computing yields
\begin{equation}
  \label{eq:c_j}
  c_j = \sum_{t \in \Upsilon_{j,1} \cup \Upsilon_{j,2}} t,
\end{equation}
where
\begin{align}
  \Upsilon_{j,1} &:= \bigcup_{
    \substack{ 1 \leqslant s < n \\ s + \alpha = \gamma_j }
  }
  \left\{
    \frac{\Beta(\alpha,1+s) \partial_1^s f(0, c_0)}
    {\Gamma(\alpha)}
  \right\},
  \label{eq:c_j-1} \\
  \Upsilon_{j,2} & := \bigcup_{s=1}^{n-1}
  \bigcup_{k=1}^s
  \bigcup_{
    \substack{
      \beta \in \Lambda_s \\ \#\beta = k \\ \Gamma_\beta \neq \emptyset
    }
  }
  \left\{
    \frac{
      \Beta\left( \alpha, 1+s-k+\sum_{l=1}^k \gamma_{i_l} \right)
      \partial_\beta f(0, c_0)
    }
    { \Gamma(\alpha) \prod_{l=1}^k \gamma_{i_l} }
    \prod_{l=1}^k{ c_{i_l} \gamma_{i_l} }: \,
    ( i_1, i_2,\dots, i_k ) \in \Xi_{\beta, j}
  \right\}.
  \label{eq:c_j-2}
\end{align}
Above, $ \Beta(\cdot,\cdot) $ denotes the standard beta function, and
\[
  \#\beta := \sum_{
    \substack{ 1 \leqslant i \leqslant s \\ \beta_i = 2 }
  } 1,
\]
\[
  \Xi_{\beta, j} := \left\{
    (i_1, i_2,\dotsc, i_{\#\beta}): \,
    \alpha + s - \#\beta + \sum_{j=1}^{\#\beta} \gamma_{i_j} =
    \gamma_j \,
  \right\},
\]
for all $ 1 \leqslant s < n $ and $ \beta \in \Lambda_s $.

To prove \cref{coro:main}, by \cref{thm:main} it suffices to show that
\cref{eq:cond-f} is equivalent to
\begin{equation}
  \label{eq:lxy}
  c_j = 0 \quad \text{ for all $ j \in \Theta $, }
\end{equation}
where
\[
  \Theta := \left\{
    1 \leqslant j \leqslant J:
    \gamma_j \not\in \mathbb N
  \right\}.
\]
But, by \cref{eq:c_j,eq:c_j-1,eq:c_j-2}, an obvious induction gives
\begin{equation}
  \label{eq:lxy-2}
  c_j = 0 \quad \text{for all $ 1 \leqslant j \leqslant J $,}
\end{equation}
if \cref{eq:cond-f} holds. Therefore, it remains to show that \cref{eq:lxy}
implies \cref{eq:cond-f}.

To this end, let us assume that \cref{eq:lxy} holds. Note that we have
\cref{eq:lxy-2}. If this statement was false, then let
\[
  j_0 := \min\left\{
    1 \leqslant j \leqslant J:
    c_j \neq 0
  \right\}.
\]
Obviously, we have $ j_0 > 1 $ and $ \gamma_{j_0} \in \mathbb N $, and in this
case, $ \Upsilon_{j_0,1} $ is empty. Thus, by \cref{eq:c_j} we have
\[
  c_{j_0} = \sum_{ t \in \Upsilon_{j_0,2} } t.
\]
But, by the definition of $ \Upsilon_{j_0,2} $ and the fact that $ c_j = 0 $ for
all $ 1 \leqslant j < j_0 $, it is straightforward that $ c_{j_0} = 0 $, which
is contrary to the definition of $ j_0 $. Therefore \cref{eq:lxy-2} holds
indeed. Using this result, from \cref{eq:c_j,eq:c_j-2} it follows
\[
  c_j = \sum_{ t \in \Upsilon_{j,1} } t \quad
  \text{for all $ 1 \leqslant j \leqslant J $,}
\]
and then, using \cref{eq:lxy-2} again, we obtain \cref{eq:cond-f}. This
completes the proof of \cref{coro:main}.

\appendix
\begin{appendices}
\section{Proofs of \texorpdfstring{ \cref{lem:lxy-3,lem:lxy-4} }{}}
\label{sec:append-proofs}
To prove \cref{lem:lxy-3}, we need the following two lemmas.
\begin{lem}
  \label{lem:lxy-11}
  Let $ h > 0 $, $ \gamma > 0 $ and $ g \in \mathcal C^1[0, h] $. We have $ w
  \in \mathcal C^1[0, h] $ and
  \begin{equation}
    w'(x) = \int_0^x t^{\gamma-1} g'(t) \, \mathrm{d}t.
  \end{equation}
  where
  \[
    w(x) := \int_0^x t^{\gamma-1} g(t) \, \mathrm{d}t,
    \quad 0 \leqslant x \leqslant h.
  \]
\end{lem}
\noindent Since the proof of this lemma is straightforward, it is omitted.

\begin{lem}
  \label{lem:lxy-2}
  Let $ 0 < h \leqslant a $, and $ l \in \mathbb N_{>0} $ such that $ l\alpha
  \leqslant 1 < (l+1)\alpha $. For any $ g \in \mathcal C^{0, l\alpha}[0, h] $, we
  have $ w \in \mathcal C^{0,(l+1)\alpha-1}[0, h] $ and
  \[
    \nm{w}_{ C^{ 0,(l+1)\alpha-1 }[0, h] } \leqslant C
    \nm{g}_{ C^{0, l\alpha}[0, h] },
  \]
  where
  \[
    w(x) := \sum_{j=1}^J \gamma_j c_j x^{\gamma_j-1} g(x),
    \quad 0 < x \leqslant h.
  \]
\end{lem}
\begin{proof}
  It suffices to prove that, for any $ 1 \leqslant j \leqslant J $, we have $ v
  \in \mathcal C[0, h] $ and
  \[
    \nm{v}_{ C^{ 0,(l+1)\alpha-1 }[0, h] } \leqslant C
    \nm{g}_{C^{0, l\alpha}[0, h]},
  \]
 where $ v(x) := x^{\gamma_j-1} g(x) $, $ 0 < x \leqslant h $. Noting the fact
 that $ l\alpha + \gamma_j > 1 $ and $ g \in \mathcal C^{0, l\alpha}[0, h] $, we
 easily obtain $ v \in \mathcal C[0, h] $ and
  \[
    \nm{v}_{C[0, h]} \leqslant C
    \nm{g}_{C^{0, l\alpha}[0, h]}.
  \]
  It remains, therefore, to prove that
  \[
    \snm{v(y) - v(x)} \leqslant C
    (y-x)^{ (l+1)\alpha-1 } \nm{g}_{ C^{0, l\alpha}[0, h] }
  \]
  for all $ 0 < x < y \leqslant h $. Moreover, since it holds
  \begin{align*}
    & \snm{v(y) - v(x)} =
    \snm{ y^{\gamma_j-1}g(y) - x^{\gamma_j-1}g(x) } \\
    ={} & \snm{
      y^{\gamma_j-1} \big(g(y)-g(x)\big) +
      (y^{\gamma_j-1} - x^{\gamma_j-1}) g(x)
    } \\
    \leqslant{} & \left(
      y^{\gamma_j-1} (y-x)^{l\alpha} +
      \snm{y^{\gamma_j-1} - x^{\gamma_j-1}} x^{l\alpha}
    \right) \nm{g}_{C^{0, l\alpha}[0, h]},
  \end{align*}
  by the fact $ g \in \mathcal C^{0, l\alpha}[0, h] $, we only need to prove that
  \begin{equation}
    \label{eq:lxy-2-1}
    y^{\gamma_j-1} (y-x)^{l\alpha} +
    \snm{ y^{\gamma_j-1} - x^{\gamma_j-1} } x^{l\alpha}
    \leqslant C (y-x)^{(l+1)\alpha-1}
  \end{equation}
  for all $ 0 < x < y \leqslant h $.

  Let us first consider the case of $ \gamma_j < 1 $. A simple algebraic
  calculation gives
  \[
    (x^{\gamma_j-1} - y^{ \gamma_j-1 }) x^{l\alpha} =
    (y-x)^{ l\alpha+\gamma_j-1 } \big(
      A^{\gamma_j-1} - (1+A)^{\gamma_j-1}
    \big) A^{l\alpha},
  \]
  where $ A := \frac x{y-x} $. If $ 0 \leqslant A \leqslant 1 $, then by the
  fact $ l\alpha + \gamma_j - 1 > 0 $ we have
  \[
    \big( A^{\gamma_j-1} - (1+A)^{\gamma_j-1} \big) A^{l\alpha}
    < A^{l\alpha + \gamma_j - 1} \leqslant 1.
  \]
  If $ A > 1 $, then using the Mean Value Theorem and the fact $ l\alpha +
  \gamma_j - 2 < 0 $ gives
  \[
    \big( A^{\gamma_j-1} - (1+A)^{\gamma_j-1} \big) A^{l\alpha}
    < (1-\gamma_j) A^{l\alpha + \gamma_j - 2} < (1-\gamma_j) < 1.
  \]
  Consequently, we obtain
  \[
    (x^{\gamma_j-1} - y^{\gamma_j-1}) x^{l\alpha} <
    (y-x)^{l\alpha+\gamma_j-1},
  \]
  which, together with the trivial estimate
  \[
    y^{\gamma_j-1}(y-x)^{l\alpha}
    < (y-x)^{\gamma_j-1} (y-x)^{l\alpha}
    = (y-x)^{l\alpha+\gamma_j-1},
  \]
  yields \cref{eq:lxy-2-1}.

  Then, since \cref{eq:lxy-2-1} is evident in the case of $ \gamma_j = 1 $, let
  us consider the case of $ 1 < \gamma_j < 2 $. Since $ 0 < \gamma_j - 1 < 1 $,
  we have
  \[
    y^{\gamma_j-1} - x^{\gamma_j-1} < (y-x)^{\gamma_j-1}.
  \]
  By the definition of $ \gamma_j $ it is clear that
  \[
    \gamma_j - 1 \geqslant (l+1)\alpha - 1.
  \]
  Using the above two estimates, we obtain
  \[
    \snm{ y^{\gamma_j-1} - x^{\gamma_j-1} } x^{l\alpha} \leqslant
    C ( y^{\gamma_j-1} - x^{\gamma_j-1} ) \leqslant
    C(y-x)^{(l+1)\alpha-1},
  \]
  which, together with the estimate
  \[
    y^{\gamma_j-1} (y-x)^{l\alpha}
    \leqslant C (y-x)^{l\alpha}
    \leqslant C (y-x)^{ (l+1)\alpha - 1 },
  \]
  indicates \cref{eq:lxy-2-1}.

  Finally, let us consider the case of $ \gamma_j \geqslant 2 $. Using the Mean
  Value Theorem gives
  \[
    \snm{ y^{\gamma_j-1} - x^{\gamma_j-1} } x^{l\alpha}
    \leqslant C (y-x)^{(l+1)\alpha-1}.
  \]
  and then, by the obvious estimate
  \[
    y^{\gamma_j-1} (y-x)^{l\alpha}
    \leqslant C(y-x)^{(l+1)\alpha-1},
  \]
  we obtain \cref{eq:lxy-2-1}. This completes the proof.
\end{proof}

\noindent {\bf Proof of \cref{lem:lxy-3}}
  Since $ g \in \mathcal C^{k, l\alpha}[0, h] $, by \cref{lem:lxy-11} we have $ w
  \in \mathcal C^k[0, h] $ and
  \begin{equation}
    \label{eq:lxy-3-1}
    w^{(i)}(x) = \int_0^x \sum_{j=1}^J \gamma_j c_j t^{\gamma_j-1}
    g^{(i)}(t) \,\mathrm{d}t, \quad i = 0, 1, 2, \dots, k.
  \end{equation}
  It follows
  \[
    \nm{w}_{C^k[0, h]} \leqslant C \nm{g}_{C^k[0, h]}.
  \]
  Therefore, it remains to prove that
  \begin{equation}
    \label{eq:lxy-3-2}
    \snm{w^{(k)}}_{ C^{ 0,(l+1)\alpha }[0, h] } \leqslant C
    \nm{g}_{ C^{k, l\alpha}[0, h] }
  \end{equation}
  if $ (l+1)\alpha \leqslant 1 $; and that $ w^{(k+1)} \in \mathcal
  C^{0,(l+1)\alpha-1}[0, h] $ and
  \begin{equation}
    \label{eq:lxy-3-3}
    \nm{w^{(k+1)}}_{ C^{ 0,(l+1)\alpha-1 }[0, h] } \leqslant C
    \nm{g}_{ C^{k, l\alpha}[0, h] }
  \end{equation}
  if $ (l+1)\alpha > 1 $.

  Let us first consider \cref{eq:lxy-3-2}. Noting the fact that $ g^{(k)} \in
  \mathcal C^{0, l\alpha}[0, h] $ and $ \gamma_j \geqslant \alpha $ for all $ 1
  \leqslant j \leqslant J $, by \cref{eq:lxy-3-1} a simple computing gives that
  \[
    \snm{ w^{(k)}(y) - w^{(k)}(x) } \leqslant C
    \snm{g^{(k)}}_{ C^{0, l\alpha}[0, h] } (y-x)^{(l+1)\alpha}
  \]
  for all $ 0 \leqslant x < y \leqslant h $, which implies \cref{eq:lxy-3-2}.

  Then let us consider \cref{eq:lxy-3-3}. Since $ g^{(k)} \in \mathcal
  C^{0, l\alpha} $, by \cref{lem:lxy-2} we have $ v \in \mathcal
  C^{0,(l+1)\alpha-1}[0, h] $ and
  \[
    \nm{v}_{C^{0,(l+1)\alpha-1}[0, h]} \leqslant C
    \nm{g^{(k)}}_{C^{0, l\alpha}[0, h]},
  \]
  where
  \[
    v(x) := \sum_{j=1}^J \gamma_j c_j x^{\gamma_j-1} g^{(k)}(x),
    \quad 0 < x \leqslant h.
  \]
  Then, by \cref{eq:lxy-3-1} we readily obtain $ w^{(k+1)} \in \mathcal
  C^{0,(l+1)\alpha-1} $ and \cref{eq:lxy-3-3}, and thus complete the proof of
  this lemma.
\hfill\ensuremath{\blacksquare}

Before proving \cref{lem:lxy-4}, let us introduce the following lemma.
\begin{lem}
  \label{lem:lxy-1}
  Let $ 0 < h \leqslant a $ and $ \gamma > 0 $. For any $ g \in \mathcal
  C^k[0, h] $ with $ 1 \leqslant k \leqslant m $, we have $ w \in \mathcal
  C^{k-1}[0, h] $ and
  \[
    \nm{w}_{ C^{k-1}[0, h] } \leqslant C \nm{g^{(k)}}_{C[0, h]},
  \]
  where
  \[
    w(x) := g(x) x^{\gamma-1}, \quad 0 < x \leqslant h,
  \]
  and $ C $ is a positive constant that only depends on $ a $, $ k $ and $
  \gamma $.
\end{lem}
\begin{proof}
  If $ k = 1 $, then, by the Mean Value Theorem and the fact $ g(0) = 0 $, this
  lemma is evident. Thus, below we assume that $ 2 \leqslant k \leqslant m $. In
  the rest of this proof, for ease of notation, the symbol $ C $ denotes a
  positive constant that only depends on $ a $, $ k $ and $ \gamma $, and its
  value may differ at each occurrence.

  Let us first show that, for $ 0 \leqslant i < k $, we have $ w_i \in \mathcal
  C[0, h] $ and
  \begin{equation}
    \label{eq:lxy-1}
    \nm{w_i}_{C[0, h]} \leqslant C \nm{g^{(k)}}_{C[0, h]},
  \end{equation}
  where
  \[
    w_i(x) := w^{(i)}(x), \quad 0 < x \leqslant h.
  \]
  To this end, let $ 0 \leqslant i < k $, and note that an elementary computing
  gives
  \begin{equation}
    \label{eq:basic-4-1}
    w_i(x) = \sum_{j=0}^i c_{ij} g^{(j)}(x) x^{ \gamma-1-i+j },
    \quad 0 < x \leqslant h,
  \end{equation}
  where $ c_{ij} $ is a constant that only depends on $ \gamma $, $ i $ and $ j
  $, for all $ 0 \leqslant j \leqslant i $.  Since $ g \in \mathcal C^k[0, h] $,
  we have $ g^{(j)} \in \mathcal C^{k-j}[0, h] $, and then, applying Taylor's
  formula with integral remainder yields
  \[
    g^{(j)}(x) = \frac1{(k-j-1)!}
    \int_0^x (x-t)^{k-j-1} g^{(k)}(t) \, \mathrm{d}t,
    \quad 0 \leqslant x \leqslant h.
  \]
  It follows that
  \begin{equation}
    \label{eq:basic-4-2}
    \snm{ g^{(j)}(x) x^{ \gamma-1-i+j } }
    \leqslant \frac{ \nm{g^{(k)}}_{C[0, h]} }{(k-j)!}
    x^{ \gamma+k-(i+1) }, \quad 0 < x \leqslant h.
  \end{equation}
  Since $ \gamma+k-(i+1) \geqslant \gamma > 0 $, this implies $ g^{(j)}(x)
  x^{\gamma-i-1+j} \in \mathcal C[0, h] $ and
  \[
    \nm{ g^{(j)}(\cdot) (\cdot)^{ \gamma-i-1+j } }_{C[0, h]}
    \leqslant C \nm{g^{(k)}}_{C[0, h]}.
  \]
  Therefore, by \cref{eq:basic-4-1} it follows $ w_i \in \mathcal C[0, h] $ and
  \cref{eq:lxy-1}.

  Then let us proceed to prove this lemma. Let $ i < k-1 $. Note that by
  \cref{eq:basic-4-1} we have
  \[
    w_i'(x) = w_{i+1}(x), \quad 0 < x \leqslant h.
  \]
  Since we have already proved that $ w_i $, $ w_{i+1} \in \mathcal C[0, h] $, by
  the Mean Value Theorem it is evident that $ w_i \in \mathcal C^1[0, h] $ and
  \[
    w_i'(x) = w_{i+1}(x), \quad 0 \leqslant x \leqslant h.
  \]
  It follows $ w_0 \in \mathcal C^{k-1}[0, h] $ and
  \[
    w_0^{(i)} = w_i, \quad 0 \leqslant i < k,
  \]
  and hence, by \cref{eq:lxy-1} we have
  \[
    \nm{w_0}_{ C^{k-1}[0, h] } \leqslant C \nm{g^{(k)}}_{C[0, h]}.
  \]
  Noting the fact $ w = w_0 $, this completes the proof.
\end{proof}
\noindent {\bf Proof of \cref{lem:lxy-4}}
  Below we employ the well-known principle of mathematical induction to prove
  this lemma. Firstly, it is clear that \cref{eq:lxy-4-1} holds in the case $ k
  = 0 $. Secondly, assuming that \cref{eq:lxy-4-1} holds for $ k = l $ where $ 0
  \leqslant l < m-1 $, let us prove that \cref{eq:lxy-4-1} holds for $ k = l+1
  $. To this end, a straightforward computing gives
  \begin{equation}
    \label{eq:lxy-4-2}
    ( \mathcal T_{w,\beta, h} z )'(x) =
    \mathcal T_{w',\beta, h} z(x) +
    \mathcal T_{ w,\widetilde \beta, h } z(x) +
    \mathcal T_{ \widetilde w,\overset{ \thickapprox }\beta, h } z(x)
  \end{equation}
  for all $ 0 < x \leqslant h $, where $ \widetilde \beta := (\beta_1, \beta_2,
  \dotsc, \beta_s, 1) $, $ \overset{\thickapprox}\beta, := (\beta_1, \beta_2,
  \dotsc, \beta_s, 2) $, and
  \[
    \widetilde w(x) := w(x) \left(
      z'(x) +
      \sum_{j=1}^J\gamma_j c_jx^{\gamma_j-1}
    \right).
  \]
  Since $ w \in \mathcal C^k[0, h] $, we have $ w' \in \mathcal C^{k-1}[0, h] $,
  and by \cref{lem:lxy-1} we have $ \widetilde w \in \mathcal C^{k-1}[0, h] $;
  consequently, $ \mathcal T_{w',\beta, h}z $ and $ \mathcal T_{ \widetilde w,
  \overset{\thickapprox} \beta, h } z $ are well-defined, and they both belong
  to $ \mathcal C[0, h] $. Therefore, by the Mean Value Theorem, and the fact $
  \mathcal T_{w,\beta, h} z \in \mathcal C[0, h] $, it follows that $ \mathcal
  T_{w,\beta, h}z \in \mathcal C^1[0, h] $, and \cref{eq:lxy-4-2} holds for all
  $ 0 \leqslant x \leqslant h $. By our assumption, we have the following
  results: $ \mathcal T_{w',\beta, h}z \in \mathcal C^{\min\{k-1, n-s\}}[0, h] $
  and
  \[
    \nm{\mathcal T_{w',\beta, h}z}_{C^{\min\{k-1, n-s\}}[0, h]}
    \leqslant C \nm{w'}_{C^{k-1}[0, h]}
    \sum_{j=0}^{\min\{k-1, n-s\}}\nm{z'}_{C^{m-1}[0, h]}^j;
  \]
  $ \mathcal T_{w,\widetilde\beta, h}z \in \mathcal C^{\min\{k-1, n-s-1\}}[0, h] $ and
  \[
    \nm{\mathcal T_{w,\widetilde\beta, h}z}_{C^{\min\{k-1, n-s-1\}}[0, h]}
    \leqslant C
    \nm{w}_{C^{k-1}[0, h]}
    \sum_{j=0}^{\min\{k-1, n-s-1\}}\nm{z'}_{C^{m-1}[0, h]}^j;
  \]
  $ \mathcal T_{\widetilde w,\overset{\thickapprox}\beta, h}z \in \mathcal
  C^{\min\{k-1, n-s-1\}}[0, h] $ and
  \[
    \nm{
      \mathcal T_{ \widetilde w,\overset{\thickapprox}\beta, h }z
    }_{ C^{\min\{k-1, n-s-1\}}[0, h] } \leqslant C
    \nm{ \widetilde w }_{ C^{k-1}[0, h] }
    \sum_{j=0}^{ \min\{k-1, n-s-1 \}}\nm{z'}_{ C^{m-1}[0, h] }^j.
  \]
  In addition, by \cref{lem:lxy-1} we easily obtain
  \[
    \nm{ \widetilde w }_{ C^{k-1}[0, h] } \leqslant C \nm{w}_{C^k[0, h]}
    \left( 1+\nm{z'}_{ C^{k-1}[0, h] } \right).
  \]
  As a consequence, we obtain $ \mathcal T_{w,\beta, h} z \in \mathcal
  C^{ \min \{ k, n-s \} } $ and
  \begin{align*}
    \nm{ ( \mathcal T_{w,\beta, h}z )' }_{ C^{\min\{k-1, n-s-1\}}[0, h] } \leqslant C
    \nm{w}_{C^k[0, h]}
    \sum_{j=0}^{ \min\{ k, n-s \} }\nm{z'}_{ C^{m-1}[0, h] }^j.
  \end{align*}
  Then \cref{eq:lxy-4-1} follows from the obvious estimate
  \[
    \nm{ \mathcal T_{w,\beta, h}z }_{C[0, h]} \leqslant C \nm{w}_{C[0, h]}.
  \]
  This completes the proof of \cref{lem:lxy-4}.
\hfill\ensuremath{\blacksquare}

\end{appendices}

\end{document}